\documentclass[12pt,a4paper]{article}
\usepackage{indentfirst}
\usepackage{amsmath,amssymb,amsfonts,amsthm,graphics}
\usepackage{makeidx}
\usepackage{ifpdf}
\newtheorem{thm}{Theorem}

\newtheorem{lem}{Lemma}

\theoremstyle{definition}

\def\-{\mbox{--}}

\newtheorem{pro}{Proposition}

\newtheorem{obser}{Observation}

\begin{document}
\title{\Large\bf On extremal graphs with exactly one\\ Steiner tree connecting any
$k$ vertices\footnote{Supported by NSFC No.11071130.}}
\author{\small  Xueliang~Li, Yan~Zhao\\
\small Center for Combinatorics and LPMC-TJKLC\\
\small Nankai University, Tianjin 300071, China\\
\small lxl@nankai.edu.cn; zhaoyan2010@mail.nankai.edu.cn}
\date{}
\maketitle
\begin{abstract}
The problem of determining the largest number
$f(n;\overline{\kappa}\leq \ell)$ of edges for graphs with $n$
vertices and maximal local connectivity at most $\ell$ was
considered by Bollob\'{a}s. Li et al. studied the largest number
$f(n;\overline{\kappa}_3\leq2)$ of edges for graphs with $n$
vertices and at most two internally disjoint Steiner trees
connecting any three vertices. In this paper, we further study the
largest number $f(n;\overline{\kappa}_k=1)$ of edges for graphs with
$n$ vertices and exactly one Steiner tree connecting any $k$
vertices with $k\geq 3$. It turns out that this is not an easy task
to finish, not like the same problem for the classical connectivity
parameter. We determine the exact values of
$f(n;\overline{\kappa}_k=1)$ for $k=3,4,n$, respectively, and
characterize the graphs which attain each of these values.



\end{abstract}

\section{Introduction}
All graphs considered in this paper are simple, finite and
undirected. We follow the terminology and notation of Bondy and
Murty \cite{Bondy}. We refer to the number of vertices in a graph as
the $order$ of the graph and the number of its edges as its $size$.
We say that two paths are \emph{internally disjoint} if they have no
common vertex except the end vertices. For any two distinct vertices
$u$ and $v$ in a graph $G$, the \emph{local connectivity}
$\kappa_{G}(u,v)$ is the maximum number of internally disjoint paths
connecting $u$ and $v$. Then the connectivity of $G$ is defined as
$\kappa(G)= \min\{\kappa_{G}(u,v):u,v\in V(G),u\neq v\}$; whereas
$\overline{\kappa}(G)= \max\{\kappa_{G}(u,v):u,v\in V(G),u\neq v\}$
is called the \emph{maximal local connectivity} of $G$, introduced
by Bollob\'{a}s.

Bollob\'{a}s \cite{Bollobas1} considered the problem of determining
the largest number $f(n;\overline{\kappa}\leq \ell)$ of edges for
graphs with $n$ vertices and maximal local connectivity at most
$\ell$. In other words, $f(n;\overline{\kappa}\leq \ell)=\max\{e(G):
|V(G)|=n\ \text{and}\ \overline{\kappa}(G)\leq \ell\}$. Determining
the exact value of $f(n;\overline{\kappa}\leq \ell)$ has got a great
attention and many results have been worked out, see
\cite{Bollobas1, Bollobas2, Leonard1, Leonard2, Leonard3, Mader1,
Mader2, Thomassen}.

For a graph $G(V,E)$ and a subset $S$ of $V$ where $|S|\geq 2$, an
\emph{$S$-Steiner tree} or a \emph{Steiner tree connecting $S$} is a
subgraph $T(V',E')$ of $G$ which is a tree such that $S\subseteq
V'$. Two $S$-Steiner trees $T_1$ and $T_2$ are called
\emph{internally disjoint} if $E(T_1)\cap E(T_2)=\varnothing$ and
$V(T_1)\cap V(T_2)=S$. Note that $T_1$ and $T_2$ are vertex-disjoint
in $G\setminus S$. For $S\subseteq V$, the \emph{generalized local
connectivity} $\kappa(S)$ is the maximum number of internally
disjoint trees connecting $S$ in $G$. The \emph{generalized
$k$-connectivity} is defined as
$\kappa_k(G)=\min\{\kappa(S):S\subseteq V(G),|S|=k\}$, which was
introduced by Chartrand et al. in 1984 \cite{Chartrand}. Some
results have been worked out on the generalized connectivity, we
refer the reader to \cite{LL, LLZ} for details.

In analogue to the classical maximal local connectivity, another
parameter $\overline{\kappa}_k(G)=\max\{\kappa(S):S\subseteq
V(G),|S|=k\}$, called the \emph{maximal generalized local
connectivity} of $G$, was introduced in \cite{lilimao}. The authors
studied the largest number $f(n;\overline{\kappa}_3\leq2)$ of edges
for graphs with $n$ vertices and at most two internally disjoint
Steiner trees connecting any three vertices.

In this paper, we will study the problem of determining the largest
number $f(n;\overline{\kappa}_k=1)$ of edges for graphs with $n$
vertices and maximal generalized local connectivity exactly equal to
$1$, that is, $f(n;\overline{\kappa}_k=1)=\max\{e(G):|V(G)|=n \
\text{and} \ \overline{\kappa}_k(G)=1\}$. It is easy to see that for
$k=2$, $f(n;\overline{\kappa}=1)=n-1$, and if a graph $G$ satisfies
$\overline{\kappa}(G)=1$, then $G$ must be a tree. It turns out that
for $k\geq 3$, the problem is not easy to attack.

This paper is organized as follows. In Section 2, we introduce a
graph operation to describe three graph classes. In Section 3, we
first estimate the exact value of $f(n;\overline{\kappa}_3=1)$, that
is, $f(n;\overline{\kappa}_3=1 )=\frac{4n-3-r}{3}$ for $n=3r+q$,
$0\leq q\leq2$. Then, in Section 4, we determine
$f(n;\overline{\kappa}_4=1)$ for $n=4r+q$, $0\leq q\leq3$. Finally,
in Section 5, $f(n;\overline{\kappa}_n=1)$ is determined to be ${n-1
\choose 2}+1$. Furthermore, we characterize extremal graphs
attaining each of these values. For general $k$, we get the lower
bound of $f(n;\overline{\kappa}_k=1)$ by constructing extremal
graphs for $n=r(k-1)+q$, $0\leq q\leq k-2$.

\section{Preliminaries}

In this section, we first give some definitions frequently used in
the sequel, and then introduce a graph operation to describe three
graph classes.

For a graph $G$, we say a path $P=u_1u_2\cdots u_q$ is an $ear$ of
$G$ if $V(G)\cap V(P)=\{u_1,u_q\}$. If $u_1\neq u_q$, $P$ is an
$open ~ear$; otherwise $P$ is a $closed~ ear$. By $\ell(P)$ we
denote the length of $P$ and $C_q$ a cycle on $q$ vertices.

Let $H_1$ and $H_2$ be two connected graphs. We obtain a graph
$H_1+H_2$ from $H_1$ and $H_2$ by joining an edge $uv$ between $H_1$
and $H_2$ where $u\in H_1$, $v\in H_2$. We call this operation the
{\it adding operation}.

$\{C_3\}^i+\{C_4\}^j+\{C_5\}^k+\{K_1\}^\ell$ is a class of connected
graphs obtained from $i$ copies of $C_3$, $j$ copies of $C_4$, $k$
copies of $C_5$ and $\ell$ copies of $K_1$ by the adding operations
such that $0\leq i\leq \lfloor\frac{n}{3}\rfloor$, $0\leq j\leq 2$,
$0\leq k\leq 1$, $0\leq \ell \leq 2$ and $3i+4j+5k+\ell=n$. Note
that these operations are taken in an arbitrary order.

Let $n=3r+q$, $0\leq q\leq2$. If $q=0$,
$\mathcal{G}_n^{0}=\{C_3\}^{r}$. If $q=1$,
$\mathcal{G}_n^{1}=\{C_3\}^{r}+K_1$ or $\{C_3\}^{r-1}+C_4$. If
$q=2$, $\mathcal{G}_n^{2}=\{C_3\}^{r}+\{K_1\}^2$ or
$\{C_3\}^{r-1}+C_4+K_1$ or $\{C_3\}^{r-1}+C_5$ or
$\{C_3\}^{r-2}+\{C_4\}^2$.

Let $A,B,D_1,D_2,D_3,F_1,F_2,F_3,F_4$ be the graphs shown in Figure
1.

\begin{figure}[h,t,b,p]
\begin{center}
\scalebox{0.9}[0.9]{\includegraphics{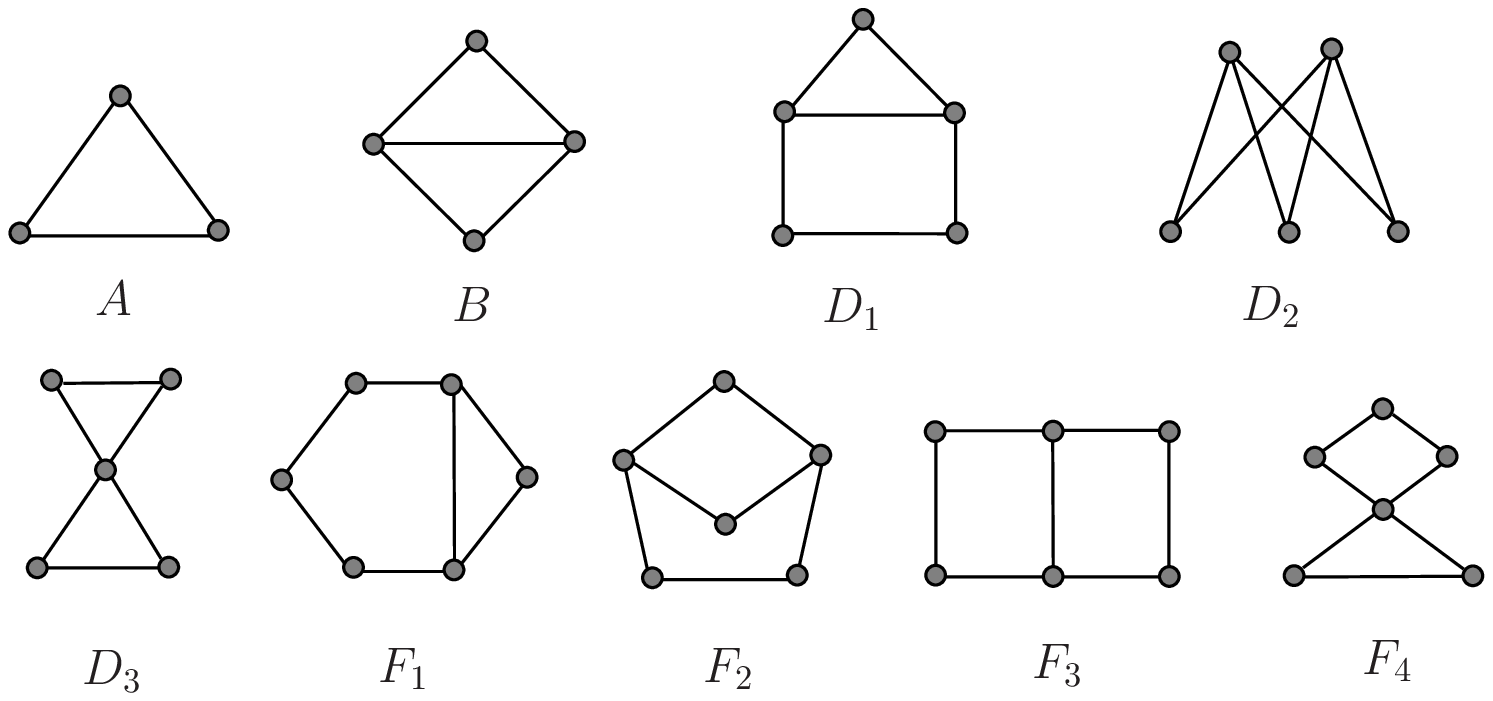}}\\[20pt]

Figure~1. The graphs used for the second graph class

\end{center}
\end{figure}

$\{A\}^{i_0}+\{B\}^{i_1}+\{D_1\}^{i_2}+\{D_2\}^{i_3}+\{D_3\}^{i_4}
+\{F_1\}^{i_5}+\{F_2\}^{i_6}+\{F_3\}^{i_7}+\{F_4\}^{i_8}+\{K_1\}^{i_9}$
is composed of another connected graph class by the adding
operations such that (1) $0\leq i_0\leq2$, $0\leq i_1\leq
\lfloor\frac{n}{4}\rfloor$, $0\leq i_2+i_3+i_4\leq2$, $0\leq
i_5+i_6+i_7+i_8\leq1$, $0\leq i_9\leq2$; (2) $D_i$ and $F_j$ are not
simultaneously in a graph belonging to this graph class where $1\leq
i\leq 3$, $1\leq j\leq 4$; (3)
$3i_0+4i_1+5(i_2+i_3+i_4)+6(i_5+i_6+i_7+i_8)+i_9=n$.

Let $n=4r+q$, $0\leq q\leq3$. If $q=0$, $\mathcal{H}_n^0=\{B\}^r$.
If $q=1$, $\mathcal{H}_n^1=\{B\}^r+K_1$ or $\{B\}^{r-1}+D_i$ ($1\leq
i\leq 3$). If $q=2$, $\mathcal{H}_n^2=\{B\}^r+\{K_1\}^2$ or
$\{B\}^{r-1}+\{A\}^2$ or $\{B\}^{r-1}+D_i+K_1$ or
$\{B\}^{r-2}+D_i+D_j$ ($1\leq i,j\leq3$) or $\{B\}^{r-1}+F_i$
($1\leq i\leq 4$). If $q=3$, $\mathcal{H}_n^3=\{B\}^r+A$.

Define the third graph class as follows: for $n=5$,
$\mathcal{K}_5=\{G:|V(G)|=5,e(G)=7\}$; for $n\geq 6$,
$\mathcal{K}_n=K_{n-1}+K_1$.

The following observation is obvious.

\begin{obser}\label{obser-1}
Let $G$ and $G^{'}$ be two connected graphs. If $G^{'}$ is a
subgraph of $G$ and $\overline{\kappa}_k(G^{'})\geq2$, then
$\overline{\kappa}_k(G)\geq2$.
\end{obser}

Next we state a famous theorem which is fundamental for calculating
the number of edge-disjoint spanning trees and getting from it a
useful lemma for our following results.

\begin{thm}(Nash-Williams \cite{Nash}, Tutte \cite{Tutte})\label{thm-1}
A multigraph contains $k$ edge-disjoint spanning trees if and only
if for every partition $\mathcal{P}$ of its vertex sets it has at
least $k(|\mathcal{P}|-1)$ cross-edges, whose ends lie in different
partition sets.
\end{thm}

\begin{lem}\label{lem-1}
Let $M$ be an edge set of $K_n$ $(n\geq 5)$ where $0\leq |M|\leq
n-3$,  and $G$ be a graph obtained from $K_n$ by deleting $M$. Then
$G$ contains two edge-disjoint spanning trees.
\end{lem}

\begin{proof}
Let $\mathcal{P}$ be a partition of $V(G)$ into $p$ sets
$V_1,V_2,\cdots,V_p$  where $1\leq p\leq n$, and let $\mathcal{E}$
represent the cross-edges. Set $|V_i|=n_i$, $1\leq i\leq p$. If
$p=1$ then this case is trivial, so we suppose next that $2\leq
p\leq n$. By Theorem \ref{thm-1}, in order to obtain two
edge-disjoint spanning trees, we only need to prove that the
inequality $|\mathcal{E}|\geq {n \choose 2}-\sum\limits_{i=1}^p{n_i
\choose 2}-|M| \geq 2(p-1)$, that is equivalent to saying that ${n
\choose 2}-|M|-2(p-1)\geq \sum\limits_{i=1}^p{n_i \choose 2}$ holds.
As $|M|\leq n-3$, and $\sum\limits_{i=1}^p{n_i \choose 2}$ attains
the maximum value ${n-p+1 \choose 2}$ by $n_i=n-(p-1)$ and $n_j=1$
where $j\neq i$. We only need to prove that ${n \choose
2}-(n-3)-2(p-1)\geq {n-p+1 \choose 2}$ holds. Let $f(n,p)={n \choose
2}-(n-3)-2(p-1)-{n-p+1 \choose 2}$. Our aim is to prove that
$f(n,p)\geq 0$. $f(n,p)={n-1 \choose 2}-2(p-2)-{n-p+1 \choose
2}=\frac{1}{2}(n-1)(n-2)-2(p-2)-\frac{1}{2}[(n-1)-(p-2)](n-p)
=\frac{1}{2}[(n-1)(p-2)+(p-2)(n-p-4)]=\frac{1}{2}(p-2)(2n-p-5)$.
Since $2\leq p\leq n$ and $n\geq 5$, it follows immediately that
$f(n,p)\geq 0$.
\end{proof}

\section {The case $k=3$}

We consider the case $k=3$ in this section. At first, we begin with
a necessary and sufficient condition for $\overline{\kappa}_3(G)=1$.

\begin{pro}\label{pro-1}
Let $G$ be a connected graph. Then $\overline{\kappa}_3(G)=1$ if and
only if every cycle in $G$ has no ear.
\end{pro}

\begin{proof}
To settle the ``only if" part, assume, to the contrary, that $C$ is
a cycle in $G$ and $P$ is an ear of $C$. Set $V(C)\cap V(P)=\{u,v\}$
where $u$ and $v$ may be the same vertex. If $\ell(P)=1$, then $P$
is an open ear, pick a vertex from $uCv$ and $vCu$ respectively, say
$u_1$ and $u_2$, $T_1=u_2Cu_1$ and $T_2=u_1Cu_2\cup uv$ are two
internally disjoint trees connecting $\{u,u_1,u_2\}$, a
contradiction to $\overline{\kappa}_3(G)=1$. If $\ell(P)\geq2$, pick
up a vertex in $C\setminus \{u,v\}$ and $P\setminus \{u,v\}$,
respectively, say $u_1$ and $u_2$, then there are also two
internally disjoint trees connecting $\{u,u_1,u_2\}$, another
contradiction.

To prove the ``if" part, let $S$ be a set of any three vertices.  We
need to prove that $\kappa_3(S)=1$. Since every cycle in $G$ has no
ear, then every maximal bridgeless subgraph of $G$ is a cycle and
each edge incident with it is a cut edge. If two vertices in $S$
belong to different cycles $C_1$ and $C_2$, then it is immediate to
check that only one tree connects $S$, since the cut edge in the
path from $C_1$ to $C_2$ can be used only once. If three vertices in
$S$ belong to a cycle, then it is immediate to see that only one
tree connects $S$. Thus $\overline{\kappa}_3(G)=1$.
\end{proof}

\begin{lem}\label{lem-2}
Let $G$ be a connected graph of order 5 and size at least 6. Then
$\overline{\kappa}_3(G)\geq2$.
\end{lem}

\begin{proof}
Let $H$ be a connected spanning subgraph of $G$ and $H$ has size
exactly 6. Since the possible connected graphs of order 5 and size 6
are $D_1$, $D_2$, $D_3$ and $B+K_1$, it is easy to see that each of
these graphs has a cycle with an ear. Then by Proposition
\ref{pro-1}, $\overline{\kappa}_3(H)\geq2$ follows. By Observation
\ref{obser-1}, it follows that $\overline{\kappa}_3(G)\geq2$.
\end{proof}

\begin{thm}\label{thm-2}
Let $n=3r+q$, where $0\leq q \leq 2$, and let $G$ be a connected
graph of order $n$ such that $\overline{\kappa}_3(G)=1$. Then
$e(G)\leq \frac{4n-3-q}{3}$, with equality if and only if $G\in
\mathcal{G}_n^q$.
\end{thm}

\begin{proof}
We apply induction on $n$. For $n=3$, $e(G)\leq 3$, and let
$G=C_3\in \mathcal{G}_n^0$. For $n=4$, if $G=K_4\setminus e$, then
there exists a cycle $C_3$ with an open ear of length 2, which
contradicts to Proposition \ref{pro-1}. Similarly, $G\neq K_4$. So
$G$ is obtained from $K_4$ by deleting two edges arbitrarily, that
is, $G=C_3+K_1$ or $C_4$, and then $G\in \mathcal{G}_n^1$. For
$n=5$, by Lemma \ref{lem-2}, $e(G)\leq 5$ and if $e(G)=5$, then
$G=C_3+\{K_1\}^2$ or $C_4+K_1$ or $C_5$, and then $G\in
\mathcal{G}_n^2$. Let $n\geq 6$. Assume that the assertion holds for
graphs of order less than $n$. We will show that the assertion holds
for graphs of order $n$. We distinguish two cases according to
whether $G$ has cut edges.

If $G$ has no cut edge, then $G$ is bridgeless, and combining with
Proposition \ref{pro-1}, $G$ is a cycle. Then
$e(G)=n<\frac{4n-5}{3}$, since $n\geq6$.

Suppose that there exists at least one cut edge in $G$. Pick up one,
say $e$. Let $G_1$ and $G_2$ be two connected components of
$G\setminus e$. Set $V(G_1)=n_1$, $V(G_2)=n_2$ where $n_1+n_2=n$.
Clearly, $e(G)=e(G_1)+e(G_2)+1$. Furthermore, set $n_1\equiv q_1 \
(mod~3)$, $n_2\equiv q_2 \ (mod~3)$ where $q_1,q_2\in\{0,1,2\}$.

If $q_1=0$ or $q_2=0$, without loss of generality, say $q_1=0$. By
induction hypothesis, $e(G_1)\leq \frac{4n_1-3}{3}$, $e(G_2)\leq
\frac{4n_2-3-q_2}{3}$. If $e(G_1)<\frac{4n_1-3}{3}$ or $e(G_2)<
\frac{4n_2-3-q_2}{3}$, then $e(G)<\frac{4n-3-q_2}{3}$. If
$e(G_1)=\frac{4n_1-3}{3}$ and $e(G_2)= \frac{4n_2-3-q_2}{3}$, then
by induction hypothesis, $G_1\in \mathcal{G}_{n_1}^0$, $G_2\in
\mathcal{G}_{n_2}^{q_2}$. It follows that $G=G_1+G_2\in
\mathcal{G}_n^{q_2}$ and $e(G)=\frac{4n-3-q_2}{3}$.

If $q_1=1$ and $q_2=1$, by hypothesis induction, $e(G_1)\leq
\frac{4n_1-4}{3}$, $e(G_2)\leq \frac{4n_2-4}{3}$. If
$e(G_1)<\frac{4n_1-4}{3}$ or $e(G_2)< \frac{4n_2-4}{3}$, then
$e(G)<\frac{4n-5}{3}$. If $e(G_1)=\frac{4n_1-4}{3}$ and
$e(G_2)=\frac{4n_2-4}{3}$, then by induction hypothesis, $G_1\in
\mathcal{G}_{n_1}^1$, $G_2\in \mathcal{G}_{n_2}^1$. It follows that
$G\in \mathcal{G}_n^2$ and $e(G)=\frac{4n-5}{3}$.

If $q_1=\{1,2\}$ and $q_2=2$, then $e(G_1)\leq \frac{4n_1-3-q_1}{3}$
and $e(G_2)\leq \frac{4n_2-5}{3}$. Thus
$e(G)\leq\frac{4n-5-q_1}{3}<\frac{4n-2-q_1}{3}$.
\end{proof}

So, we get the following result for $k=3$.
\begin{thm}\label{cor-1}
$f(n;\overline{\kappa}_3=1)=\frac{4n-3-q}{3}$, where $n=3r+q$ and
$0\leq q \leq 2$.
\end{thm}

\section{The case $k=4$}

In this section, we turn to consider the case that $k=4$. Similarly,
we will give a necessary and sufficient condition for
$\overline{\kappa}_4(G)=1$. First of all, we begin with a claim
useful for simplifying our argument. Let $P_1=u_1w_1w_2\cdots
w_kv_1$ be an ear of a cycle $C$. Set $H=C\cup P_1$ and add another
ear $P_2=u_2w_1^{'}w_2^{'}\cdots w_l^{'}v_2$ to $H$. We claim that
there is always a cycle $C^{'}$ in $H\cup P_2$ which has two ears.
If $u_2,v_2\in V(C)$, then $C^{'}=C_1^{*}$. If $u_2,v_2\in V(P_1)$,
then $C^{'}=C_2^{*}$. If $u_2\in v_1Cu_1$, $v_2\in V(P_1)$ and $P_1$
is an open ear, then $C^{'}=C_3^{*}$. If $u_2\in v_1Cu_1$, $v_2\in
V(P_1)$ and $P_1$ is a closed ear, then $C^{'}=C_4^{*}$. $C_i^{*}$
is shown in Figure 2 for $1\leq i\leq 4$.

\begin{figure}[h,t,b,p]
\begin{center}
\scalebox{1}[1]{\includegraphics{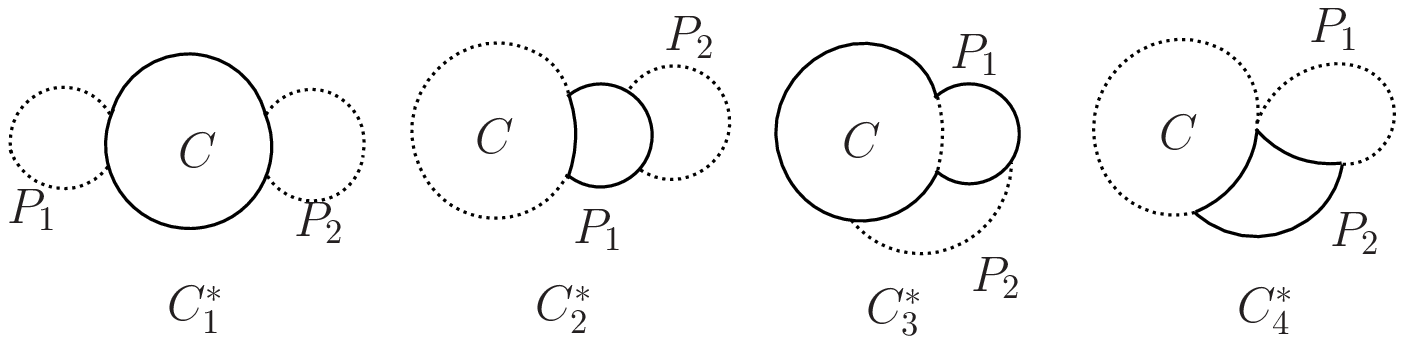}}\\[20pt]

Figure~2. $C_i^{*}$
\end{center}
\end{figure}

\begin{pro}\label{pro-2}
Let $G$ be a connected graph. Then $\overline{\kappa}_4(G)=1$ if and
only if every cycle in $G$ has at most one ear.
\end{pro}

\begin{proof}
To settle the ``only if" part, let $C$ be a cycle in $G$. Assume, to
the contrary, that $C$ has two ears $P_1$ and $P_2$. In Figure 3, we
list all cases that $C$ has two ears. The marked dots are the chosen
four vertices, and different trees are marked with different lines.
Note that if an ear $P$ of $C$ has length 1, then it together with a
segment of $C$ forms a cycle, and we can replace it with the
according segment such that an ear of a cycle has length at least 2.
From Figure 3, we can find two internally disjoint trees connecting
four vertices in $G$, a contradiction.

\begin{figure}[h,t,b,p]
\begin{center}
\scalebox{0.9}[0.9]{\includegraphics{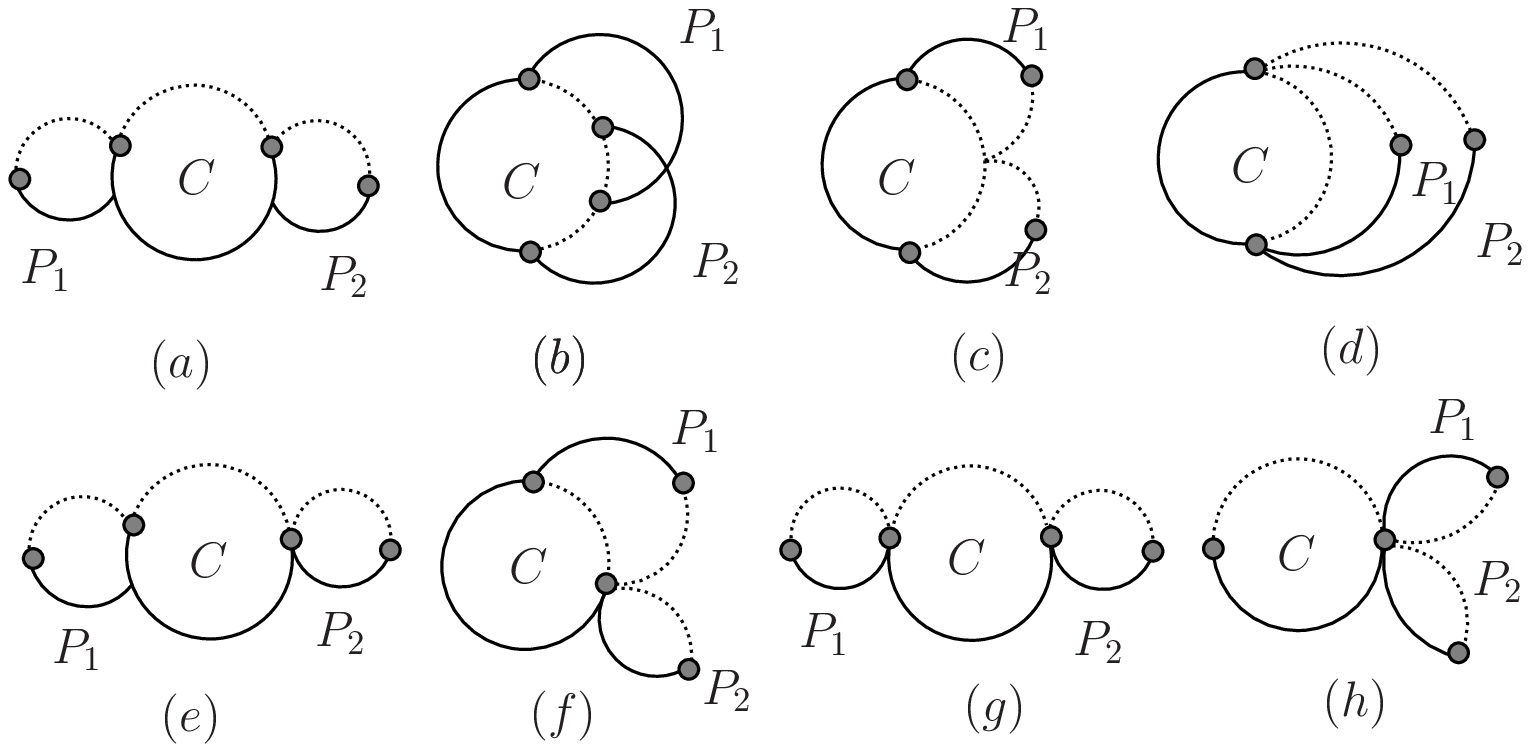}}\\[20pt]

Figure~3. Graphs for Proposition 2
\end{center}
\end{figure}

To prove the ``if" part, since every maximal bridgeless subgraph of
$G$  is a cycle $C$ or $C\cup P$, where $P$ is an ear of $C$, then
every edge incident to a maximal bridgeless subgraph of $G$ is a cut
edge of $G$. Similar to Proposition \ref{pro-1}, it is easy to check
that only one tree connects every four vertices in $G$, and so
$\overline{\kappa}_4(G)=1$.
\end{proof}

\begin{lem}\label{lem-4}
Let $G$ be a connected graph of order 5 and size 6. Then $G\in
\{B+K_1,D_1,D_2,D_3\}$ and $\overline{\kappa}_4(G)=1$.
\end{lem}

\begin{proof}
We can easily get that $\delta(G)\leq 2$; otherwise $e(G)\geq
\frac{3n}{2}=\frac{15}{2}$. If $\delta(G)=1$, by deleting a vertex
of degree one, say $v$, we obtain a graph $G^{*}=K_4\setminus e$.
Observe that $G^{*}+K_1$ has no cycle with two ears. Thus by
Proposition \ref{pro-2} $\overline{\kappa}_4(G)=1$.

Suppose that $\delta(G)=2$, without loss of generality, let
$d(v)=2$. Then $G\setminus v$ is $C_4$ or $C_3+K_1$. Adding $v$
back, there are four graphs $D_1$, $D_2$, $D_3$ or $B+K_1$, and for
each of the graphs, $\overline{\kappa}_4(G)=1$.
\end{proof}

\begin{lem}\label{lem-3}
Let $G$ be a connected graph of order 5 and size at least 7. Then
$\overline{\kappa}_4(G)\geq2$.
\end{lem}

\begin{proof}
By Lemma \ref{lem-1}, we need to check the case that $G$ has order 5
and size exactly 7. First, similar to Lemma \ref{lem-4},
$\delta(G)\leq 2$. Suppose that $\delta(G)=1$, without loss of
generality, let $d(v)=1$. Then $|V(G\setminus v)|=4$ and
$e(G\setminus v)=6$, which implies that $G\setminus v$ is $K_4$.
Then there are two internally disjoint trees connecting the four
vertices of the clique $K_4$. It follows that
$\overline{\kappa}_4(G\setminus v)\geq2$, and hence
$\overline{\kappa}_4(G)\geq2$.

If $\delta(G)=2$, suppose that $v$ has degree 2, then $|V(G\setminus
v)|=4$ and $e(G\setminus v)=5$, giving that $G\setminus v$ is
$K_4\setminus e$. Adding $v$ again, the graph $G$ has a cycle with
two ears, and by Proposition \ref{pro-2},
$\overline{\kappa}_4(G)\geq2$.
\end{proof}

\begin{lem}\label{lem-6}
Let $G$ be a connected graph of order 6 and size 7. Then $G\in
\{B+\{K_1\}^2,F_1,F_2,\\F_3,F_4\}$ and $\overline{\kappa}_4(G)=1$.
\end{lem}

\begin{proof}
Obviously, $\delta(G)\leq 2$. If $\delta(G)=1$, by deleting a vertex
of degree one we get the graphs in Lemma \ref{lem-4}. Adding $v$
again, it is easy to check that $\overline{\kappa}_4(G)=1$.

If $\delta(G)=2$, without loss of generality, let $d(v)=2$, then
$|V(G\setminus v)|=5$ and $e(G\setminus v)=5$. Then $G\setminus v$
is $C_5$ or $C_4+K_1$ or $K_3+\{K_1\}^2$. Adding $v$ again, the
graph $G$ belongs to $\{B+\{K_1\}^2,F_1,F_2,F_3,F_4\}$, and for each
of the graphs, it is easy to check that $\overline{\kappa}_4(G)=1$.
\end{proof}

\begin{lem}\label{lem-5}
Let $G$ be a connected graph of order 6 and size at least 8. Then
$\overline{\kappa}_4(G)\geq2$.
\end{lem}

\begin{proof}
We can easily get that $\delta(G)\leq 2$; otherwise $e(G)\geq
\frac{3n}{2}=9$. If $\delta(G)=1$, we delete a vertex of degree one
to get a graph of order 5 and size at least 7. Then by Lemma
\ref{lem-3}, it follows that $\overline{\kappa}_4(G)\geq2$.

If $\delta(G)=2$, without loss of generality, let $d(v)=2$, then
$|V(G\setminus v)|=5$ and $e(G\setminus v)\geq6$. If $e(G\setminus
v)\geq7$, by Lemma \ref{lem-3}, $\overline{\kappa}_4(G\setminus
v)\geq2$, and then $\overline{\kappa}_4(G)\geq2$. So we remain to
check the case $|V(G\setminus v)|=5$ and $e(G\setminus v)=6$, which
implies that $G\setminus v$ is one of the graphs in Lemma
\ref{lem-4}. Adding $v$ again, there is a cycle with two ears, and
by Proposition \ref{pro-2}, $\overline{\kappa}_4(G)\geq2$.
\end{proof}

\begin{thm}\label{thm-3}
Let $n=4r+q$, where $0\leq q \leq 3$, and let $G$ be a connected
graph of order $n$ such that $\overline{\kappa}_4(G)=1$. Then
$$
e(G)\leq \left\{
\begin{array}{cc}
\frac{3n-2}{2}&if~q=0,\\
\frac{3n-3}{2}&if~q=1,\\
\frac{3n-4}{2}&if~q=2,\\
\frac{3n-3}{2}&if~q=3.
\end{array}
\right.
$$
with equality if and only if $G\in \mathcal{H}_n^q$.
\end{thm}

\begin{proof}
We apply induction on $n$. For $n=4$, it is easy to see that
$e(G)\leq5$ and if $e(G)=5$, and then $G=B\in\mathcal{H}_n^0$. For
$n=5$, if $G$ is a connected graph of order 5 and size at least 7,
then $\overline{\kappa}_4(G)\geq2$ by Lemma \ref{lem-3}. In other
cases, either $e(G)\leq5$ or $G\in \mathcal{H}_n^1$ by Lemma
\ref{lem-4}. For $n=6$, if $G$ is a connected graph of order 6 and
size at least 8, then $\overline{\kappa}_4(G)\geq2$ by Lemma
\ref{lem-5}. In other cases, either $e(G)\leq6$ or $G\in
\mathcal{H}_n^2$ by Lemma \ref{lem-6}. Let $n\geq7$, and suppose
that the assertion holds for graphs of order less than $n$. We show
that the assertion holds for graphs of order $n$. We divide into two
cases according to whether $G$ has cut edges.

If $G$ has no cut edge, then $G$ is bridgeless, and combining with
Proposition \ref{pro-2}, $G$ is a cycle or a cycle with an ear. If
$G$ is a cycle, then $e(G)=n<\frac{3n-4}{2}$, since $n\geq7$. If $G$
is a cycle with an ear, then $e(G)=n+1<\frac{3n-4}{2}$, since
$n\geq7$.

Suppose that $G$ has cut edges. Without loss of generality, let $e$
be a cut edge. Let $G_1$ and $G_2$ be two connected components of
$G\setminus e$. Set $V(G_1)=n_1$, $V(G_2)=n_2$ where $n_1+n_2=n$.
Clearly, $e(G)=e(G_1)+e(G_2)+1$. Furthermore, set $n_1\equiv q_1 \
(mod~4)$, $n_2\equiv q_2 \ (mod~4)$ where $q_1,q_2\in\{0,1,2,3\}$.

If $q_1=0$, $q_2\in\{0,1,2\}$ or $q_1=1$, $q_2=1$, by induction
hypothesis, $e(G_1)\leq \frac{3n_1-2-q_1}{2}$, $e(G_2)\leq
\frac{3n_2-2-q_2}{2}$. If $e(G_1)<\frac{3n_1-2-q_1}{2}$ or
$e(G_2)<\frac{3n_2-2-q_2}{2}$, then $e(G)<\frac{3n-2-q_1-q_2}{2}$.
If $e(G_1)=\frac{3n_1-2-q_1}{2}$ and $e(G_2)= \frac{3n_2-2-q_2}{2}$,
then by induction hypothesis, $G_1\in \mathcal{H}_{n_1}^{q_1}$,
$G_2\in \mathcal{H}_{n_2}^{q_2}$, and it follows that $G=G_1+G_2\in
\mathcal{H}_n^{q_1+q_2}$ and $e(G)=\frac{3n-2-q_1-q_2}{2}$.

If $q_1=0$, $q_2=3$, by induction hypothesis, $e(G_1)\leq
\frac{3n_1-2}{2}$, $e(G_2)\leq \frac{3n_2-3}{2}$. If
$e(G_1)<\frac{3n_1-2}{2}$ or $e(G_2)<\frac{3n_2-3}{2}$, then
$e(G)<\frac{3n-3}{2}$. If $e(G_1)=\frac{3n_1-2}{2}$ and $e(G_2)=
\frac{3n_2-3}{2}$, then by induction hypothesis, $G_1\in
\mathcal{H}_{n_1}^0$, $G_2\in \mathcal{H}_{n_2}^3$, and it follows
that $G=G_1+G_2\in \mathcal{H}_n^3$ and $e(G)=\frac{3n-3}{2}$.

If $q_1=1$, $q_2=2$, then $e(G_1)\leq \frac{3n_1-3}{2}$ and
$e(G_2)\leq \frac{3n_2-4}{2}$, and thus $e(G)\leq
\frac{3n-5}{2}<\frac{3n-3}{2}$.

If $q_1=1$, $q_2=3$, then $e(G_1)\leq \frac{3n_1-3}{2}$, $e(G_2)\leq
\frac{3n_2-3}{2}$, and so $e(G)\leq \frac{3n-4}{2}<\frac{3n-2}{2}$.

If $q_1=2$, $q_2=2$, then $e(G_1)\leq \frac{3n_1-4}{2}$, $e(G_2)\leq
\frac{3n_2-4}{2}$, and it follows that $e(G)\leq
\frac{3n-6}{2}<\frac{3n-3}{2}$.

If $q_1=2$, $q_2=3$, then $e(G_1)\leq \frac{3n_1-4}{2}$, $e(G_2)\leq
\frac{3n_2-3}{2}$, and so $e(G)\leq \frac{3n-5}{2}<\frac{3n-3}{2}$.

If $q_1=3$, $q_2=3$, by induction hypothesis, $e(G_1)\leq
\frac{3n_1-3}{2}$, $e(G_2)\leq \frac{3n_2-3}{2}$. If
$e(G_1)<\frac{3n_1-3}{2}$ or $e(G_2)<\frac{3n_2-3}{2}$, then
$e(G)<\frac{3n-4}{2}$. If $e(G_1)=\frac{3n_1-3}{2}$ and $e(G_2)=
\frac{3n_2-3}{2}$, then by induction hypothesis, $G_1\in
\mathcal{H}_{n_1}^3$, $G_2\in \mathcal{H}_{n_2}^3$, and it follows
that $G=G_1+G_2\in \mathcal{H}_n^2$ and $e(G)=\frac{3n-4}{2}$.
\end{proof}

So, we get the following result for $k=4$.
\begin{thm}\label{cor-2}
$$
f(n;\overline{\kappa}_4=1)= \left\{
\begin{array}{cc}
\frac{3n-2}{2}&if~q=0,\\
\frac{3n-3}{2}&if~q=1,\\
\frac{3n-4}{2}&if~q=2,\\
\frac{3n-3}{2}&if~q=3,
\end{array}
\right.
$$
where $n=4r+q$ and $0\leq q \leq 3$.
\end{thm}

\section{The case $k=n$}

Let us turn now to the case $k=n$. Let $n\geq 5$, since $k=3$ and
$k=4$ have been considered before. Observe that in this case the
edge disjoint trees are the same as the internally disjoint trees.

\begin{thm}\label{thm-4}
Let $G$ be a connected graph of order $n$ such that
$\overline{\kappa}_n(G)=1$ where $n\geq 5$. Then $e(G)\leq{n-1
\choose 2}+1$, and with equality if and only if $G\in
\mathcal{K}_n$.
\end{thm}

\begin{proof}
Let $G=K_5\setminus M$, where $M$ is an edge set. On one hand, to
make $\overline{\kappa}_5(G)=1$, by Lemma \ref{lem-1} $M$ should
contain at least 3 edges, and then $e(G)\leq 7$. On the other hand,
to form two edge-disjoint spanning trees, $G$ should contain at
least 8 edges, since each tree consists of at least 4 edges. Thus a
graph with order 5 and size 7 belongs to $\mathcal{K}_5$. It
suffices to verify the case $n\geq 6$. By Lemma \ref{lem-1} again,
to make $\overline{\kappa}_n(G)=1$, $e(G)\leq {n \choose
2}-(n-2)={n-1 \choose 2}+1$.

Now we show that $\mathcal{K}_n$ contains only one graph
$K_{n-1}+K_1$. Suppose $H$ is a graph with order $n$, size ${n-1
\choose 2}+1$ and $\overline{\kappa}_n(H)=1$ but different from
$K_{n-1}+K_1$.

We claim that $2\leq \delta(H) \leq n-3$. Otherwise, if
$\delta(H)=1$, then $H=K_{n-1}+K_1$. If $\delta(H) \geq n-2$, then
$e(H)\geq \frac{(n-2)n}{2}$, $H$ is obtained from $K_n$ by deleting
at most $\frac{n}{2}$ edges. Since $n\geq 6$, then $\frac{n}{2} \leq
n-3$. By Lemma \ref{lem-1}, $H$ has two edge-disjoint spanning
trees, a contradiction.

Let $v$ be a vertex of $H$ with degree equal to $\delta(H)$, and let
$H^{*}=H\setminus v$. Since there are $n-1-d(v)$ vertices not
adjacent to $v$ in $H$ and $H$ is obtained from $K_n$ by deleting
$n-2$ edges, $H^{*}$ is obtained from $K_{n-1}$ by deleting
$n-2-(n-1-d(v))=d(v)-1\leq (n-1)-3$ edges. By Lemma \ref{lem-1},
$H^{*}$ has two edge-disjoint spanning trees $T_1^{*}$ and
$T_2^{*}$. By adding an edge incident with $v$ to $T_1^{*}$ and
$T_2^{*}$ respectively, we will obtain two edge-disjoint spanning
trees of $H$, a contradiction. Thus $\mathcal{K}_n$ contains only
one graph $K_{n-1}+K_1$.
\end{proof}

So, we get the following result for $k=n$.
\begin{thm}\label{cor-3}
$f(n;\overline{\kappa}_n=1)={n-1 \choose 2}+1$ where $n\geq 5$.
\end{thm}

\noindent \textbf{Remark}: Let $G$ be a connected graph. For $k=3$
and $k=4$, we get the necessary and sufficient conditions for
$\overline{\kappa}_k(G)=1$ by means of the number of ears of cycles.
Naturally, one might think that this method can always be applied
for $k=5$, i.e., every cycle in $G$ has at most two ears, but
unfortunately we found a counterexample: Let $G$ be a graph which
contains a cycle with three independent closed ears. Set
$C=u_1u_2u_3$, $P_1=u_1v_1w_1u_1$, $P_2=u_2v_2w_2u_2$, and
$P_3=u_3v_3w_3u_3$. Then, $\overline{\kappa}_5(G)=1$. In fact, let
$S$ be the set of chosen five vertices. Obviously, for each $i$, if
$v_i$ and $w_i$ are in $S$, then $\overline{\kappa}_5(S)=1$. So,
only one vertex in $P_i\setminus u_i$ can be chosen. Suppose that
$v_1$, $v_2$, $v_3$ have been chosen. By symmetry, $u_1$, $u_2$ are
chosen. It is easy to check that there is only one tree connecting
$\{u_1, u_2, v_1,v_2,v_3\}$. The remaining case is that all $u_1$,
$u_2$ and $u_3$ are chosen. Then, no matter which are the another
two vertices, only one tree can be found.

For general $k$ with $5\leq k \leq n-1$, we can only get the
following lower bound of $f(n;\overline{\kappa}_k=1)$. The exact
value is not easy to obtained.

\begin{thm}\label{thm-5}
$$
f(n;\overline{\kappa}_k=1)\geq
\left\{
         \begin{array}{ll}
           r{k-1 \choose 2}+r-1, & \hbox{if~$q=0$;} \\
           r{k-1 \choose 2}+{q \choose 2}+r, & \hbox{if~$1\leq q\leq k-2$.}
         \end{array}
       \right.
$$
for $n=r(k-1)+q$, $0\leq q\leq k-2$.
\end{thm}
\begin{proof}
If $q=0$, let $G=\{K_{k-1}\}^r$, then $e(G)=r{k-1 \choose 2}+r-1$.
If $1\leq q\leq k-2$, let $G=\{K_{k-1}\}^r+K_q$, and then
$e(G)=r{k-1 \choose 2}+{q \choose 2}+r$. In every case, it is easy
to verify that $\overline{\kappa}_k(G)=1$.
\end{proof}

\end{document}